\documentclass[runningheads,a4paper]{llncs}
\usepackage{llncsdoc}
\usepackage{amssymb,amsmath}
\usepackage{url}
\usepackage{graphicx}
\usepackage{verbatim}
\usepackage{lscape}
\usepackage{subfigure}
\usepackage{appendix}
\usepackage[colorlinks=true,urlcolor=red,citecolor=blue,linkcolor=red, bookmarks=true]{hyperref}

\usepackage{enumitem}

\usepackage{subfigure}
\usepackage{tikz}
\usepackage[ruled]{algorithm2e}
\usepackage{comment}

\usetikzlibrary{arrows,shapes}
\usetikzlibrary{snakes}

  \newtheorem{defi}{Definition} 

  \newtheorem{ex}{Example} 

\usepackage{natbib}
 \bibpunct[, ]{[}{]}{,}{n}{}{,}%

\newcommand{\blank}[1]{\text{[MATH]}}

\newcommand{\vect}[1]{{\boldsymbol #1 }}
\def\bx{\vect{x}}
\def\bz{\vect{z}}
\def\by{\vect{y}}
\def\be{\vect{e}}

\def\bnu{\vect{\nu}}

\def\I{\mathcal{I}}
\def\S{\mathcal{S}}
\def\N{{\mathcal{N}}}
\def\F{{\mathcal{F}}}
\def\T{\mathcal{T}}
\def\H{\mathcal{H}}

\def\Q{\mathcal{R}}

\def\C{\mathcal{C}}

\newcommand{\norm}[1]{\|#1\|}

\def\R{{\mathbb R}}

\newcommand{\mb}[1]{\mbox{\boldmath $#1$}}

\newcommand{\ba}{\begin{array}}
\newcommand{\ea}{\end{array}}

\newcommand{\BigO}[1]{\ensuremath{\operatorname{O}\bigl(#1\bigr)}}

\newcommand{\argmax}{\mathop{\rm argmax}}

\newcommand{\refs}[1]{$(\ref{#1})$}
\def\bnu{\vect{\nu}}

\renewcommand{\qed}{\hfill\blacksquare}

\begin{document}
\title{Simplifying the Kohlberg Criterion on the Nucleolus}  
\author{Tri-Dung Nguyen}

\institute{Mathematical Sciences and Business School, University of Southampton, Southampton SO17 1BJ, United Kingdom \\
\email{ T.D.Nguyen@soton.ac.uk}\\}

\titlerunning{Simplifying Kohlberg Criterion on the Nucleolus}
\maketitle                 

\begin{abstract}
The nucleolus offers a desirable payoff-sharing solution in cooperative games thanks to its attractive properties - it always exists and lies in the core (if the core is non-empty), and is unique. Although computing the nucleolus is very challenging, the Kohlberg criterion offers a powerful method for verifying whether a solution is the nucleolus in relatively small games (i.e., the number of players $n \leq 20$). This, however, becomes more challenging for larger games because of the need to form and check the balancedness of possibly exponentially large collections of coalitions, each collection could be of an exponentially large size. We develop a simplifying set of the Kohlberg criteria that involves checking the balancedness of at most $(n-1)$ sets of coalitions. We also provide a method for reducing the size of these sets and a fast algorithm for verifying the balancedness.


\end{abstract}

\section{Introduction}
Cooperative games model situations where players can form coalitions to jointly achieve some objective. Once such example is where entrepreneurs, with possibly complementary skills, consider running a business together. Assuming that the entrepreneurs can jointly run a more successful business than by working individually (or in smaller groups), a natural question is how to divide the reward among the players in such a way that could avoid any subgroup of players to breaking away from the grand coalition in order to increase the total payoff. Solution concepts in cooperative games provide the means to achieving this.

Formally, let $n$ be the number of players and $\N = \{1,2,\ldots,n\}$ be the set of all the players. A \emph{coalition} $\S$ is a subset of the players; i.e., $\S \subseteq \N$. 
 The \emph{characteristic function} $v~:~2^\N \mapsto \R$ maps each coalition to a real number with $v(\S)$ representing the payoff that coalition $\S$ is guaranteed to receive if all players in $\S$ collaborate, despite the actions of the other players. A solution (also called a payoff distribution) of the game $\bx = (x_1,x_2,\ldots,x_n)$ is a way to distribute the reward among the players, with $x_i$ being the share for player $i$.

Given the total payoff $v(\N)$ of the grand coalition, we are interested in \emph{efficient} solutions $\bx$ which satisfy $\sum_{i \in \N} x_i = v(\N)$. Let us denote $\bx(\S) = \sum_{i \in \S} x_i$. For each imputation $\bx$, the \emph{excess value} of a coalition $\S$ is defined as $d(\S, \bx) := v(\S) - \bx(\S)$ which can be regarded as the level of dissatisfaction the players in coalition $\S$ feel over the proposed solution $\bx$. Here, we concern with profit games and assume that it is more desirable to have higher shares. All the results can be extended to cost games either through transforming the characteristic function to the corresponding profit games or by redefining the excess values.

Player $i$ is considered rational if he/she only accepts a share $x_i$ of at least the amount $v(\{i\})$. A group of players $\S$ is considered rational if it only accepts a total share $\bx(\S) := \sum_{i \in \S} x_i$ of at least the amount $v(\S)$ that the group is guaranteed to receive by breaking away from the grand coalition and forming its own coalition; i.e., $d(\S, \bx) \leq 0$.

An \emph{imputation} is an efficient solution that satisfies \emph{individual rationality}; that is, $x_i \geq v(\{i\}), \forall i \in \N$. The \emph{core} of the game is the set of all efficient solutions $\bx$ such that no coalition has the incentive to break away -- i.e., satisfying \emph{group rationality} -- and hence the solutions are stable. It is, however, possible that there is no solution satisfying this set of conditions, and the core might not exist. In that case, we consider alternative solutions that, although not stable, are least susceptive to deviations. The first such solution concept is called the \emph{least core}, which minimizes the worst level of dissatisfaction among all the coalitions. Note that the least core always exists but might not be unique. We denote $\mathbf{I}$ as the imputation set and $\mathbf{Co}$ as the core of the game.

Among all solutions in the least core, if we also ensure not only the worst dissatisfaction level but also all the dissatisfactions to be lexicographically minimized, we arrive at the concept of the \emph{nucleolus} which is the `most stable' solution in the imputation set. Formally, for any imputation $\bx$, let $\Theta(\bx) =(\Theta_1(\bx),\Theta_2(\bx),\ldots,\Theta_{2^n}(\bx))$ be the vector of all the $2^n$ excess values at $\bx$ sorted in a non-increasing order; i.e., $\Theta_i(\bx) \geq \Theta_{i+1}(\bx)$ for all $1 \leq i < 2^n$. Let us denote $\Theta(\bx) <_L \Theta(\by)$ if there exists $r \leq 2^n$ such that $\Theta_i(\bx) = \Theta_i(\by),\forall 1 \leq i < r$ and $\Theta_r(\bx) < \Theta_r(\by)$. Then $\bnu \in \mathbf{I}$ is the \emph{nucleolus} if $\Theta(\bnu) <_L \Theta(\bx),~\forall \bx \in \mathbf{I}, ~\bx \neq \bnu$. If we relax the condition $\bx, \bnu \in \mathbf{I}$, we arrive at the definition of the prenucleolus.

The nucleolus is one of the most important solution concepts for cooperative games with transferable utilities, and 
was introduced in 1969 by \citet{schmeidler1969nucleolus} as a solution concept with attractive properties - it always exists (if the imputation is non-empty), it is unique, and it lies in the core (if the core is non-empty). Despite the desirable properties that the nucleolus has, its computation is, however, very challenging because the process involves the lexicographical minimization of $2^n$ excess values, where $n$ is the number of players. There are a small number of games whose nucleoli can be computed in polynomial time (e.g., \citet{solymosi1994algorithm}, \citet{hamers2003nucleolus}, \citet{solymosi2005computing}, \citet{potters2006nucleolus,deng2009finding,kern2009core}). It has been shown that finding the nucleolus is NP-hard for many classes of games such as the utility games with non-unit capacities (\citet{deng2009finding}) and the weighted voting games (\citet{elkind2007computational}).


\citet{Kopelowitz1967Computation} suggests using nested linear programming (LP) to compute the kernel of a game. This encouraged a number of researchers to study the computation of the nucleolus using linear programming. For example, \citet{kohlberg1972nucleolus} presents a single LP with $\BigO{2^n!}$ constraints which later on is improved by \citet{owen1974note} with $\BigO{4^n}$ constraints (at the cost of having larger coefficients). \citet{puerto2013finding} recently introduces a different single LP formulation with $\BigO{4^n}$ constraints and $\BigO{4^n}$ decision variables and with coefficients in $\{-1,0,1\}$. The nucleolus can also be found by solving a sequence of LPs. However, either the number of LPs involved is exponentially large (\citet{maschler1979geometric}, \citet{sankaran1991finding}) or the sizes of the LPs are exponential (\citet{potters1996computing}, \citet{derks1996implementing}, \citet{Nguyen2016}).

While finding the nucleolus is very difficult as shown in the aforementioned literature, \citet{kohlberg1971nucleolus} provides a necessary and sufficient condition for a given imputation to be the nucleolus as is described in the subsequent section. This set of criteria is particularly useful in relatively small games (e.g., less than 10 players) or in larger games with special structures which allow us to take an educated guess on the nucleolus. The verification of the criterion, however, becomes time consuming when the number of players exceeds 15, and becomes almost impossible in general cases when the number of players exceeds 25. This is because the criterion requires forming the sets of coalitions of all $2^n$ possible coalitions and iteratively verifying if unions of these sets are balanced. This work aims to resolve these issues and proposes a new set of simplifying criteria.

The key contributions of our work include the following:
\begin{itemize}
\item We present a new set of necessary and sufficient conditions for a solution to be the nucleolus in Section~\ref{subsec:simplifiedKohlberg}. The number of subsets of coalitions to check for balancedness is at most $(n-1)$ (instead of exponentially large). 
    \item The balancedness condition is essentially equivalent to solving a linear program with strict inequalities which are often undesirable in mathematical programming. We provide a solution to this in Section~\ref{subsec:balancedness_checking}.
    \item On checking the Kohlberg criterion, we might end up having to store an exponentially large number of coalitions. We provide a method for reducing this to the size of at most $n(n-1)$ in Section~\ref{subsec:simplifiedKohlberg2}.
\end{itemize}


\section{The Kohlberg criterion for verifying the nucleolus}\label{subsec:Kohlberg}
For each efficient payoff distribution $\bx$, \citet{kohlberg1971nucleolus} first defines the following sets of coalitions: $T_0(\bx) = \{\{i\},i=1,\ldots,n~:~x_i = v(\{i\})\}$, $H_0(\bx) = \{\N,\emptyset\}$ and $H_k(\bx) = H_{k-1}(\bx) \cup T_k(\bx),k=1,2,\ldots,$ where for each $k \geq 1$,
\begin{eqnarray}
  T_k(\bx) &=& \displaystyle \argmax_{\S \not \in H_{k-1}(\bx)} \left\{v(\S) - x(\S)\right\},\quad \epsilon_k(\bx) = \displaystyle \max_{\S \not \in H_{k-1}(\bx)} \left\{v(\S) - x(\S)\right\}. \nonumber
\end{eqnarray}

Here, $T_k(\bx)$ includes all coalitions that have the same excess value $\epsilon_k(\bx)$ and $\epsilon_1(\bx) > \epsilon_2(\bx) >\ldots$. The terms `collection of coalitions' and `subset of the powerset $2^\N$' are equivalent and used interchangeably in this paper. We also use the terms `collection' and `subset' as their shorter versions.

For each collection $Q \subseteq 2^\N$, let us denote $|Q|$ as the size of $Q$. We associate each collection $Q$ with a weight vector in $\R^{|Q|}$ with each element denoting the weight of the corresponding coalition in $Q$. Throughout this paper, we use bold fonts for vectors and normal font for scalars.

Let us denote $\be(\S),\S \in \N$, as a binary vector in $\R^n$ with the $i$th element equal to one if and only if player $i$ is in the coalition. With this, for all $\bx \in \R^n$, we have $\bx(\S) = \sum_{i \in \S} x_i = \bx^T \be(\S)$. The concept of balancedness is defined as follows:

\begin{defi}
A collection of coalitions $Q\subseteq 2^N$ is balanced if there exists a weight vector $\mb{\omega} \in \R^{|Q|}_{>0}$ such that $\be(\N) = \sum_{\S \in Q} \omega_\S \be(\S)$.
\end{defi}

\begin{defi}
Given a collection $T_0 \subseteq 2^N$, a collection $Q \subseteq 2^N$ is called $T_0$-balanced if there exist weight vectors $\mb{\gamma} \in \R^{|T_0|}_{\geq 0}$ and $\mb{\omega} \in \R^{|Q|}_{>0}$ such that $\be(\N) = \sum_{\S \in T_0} \gamma_S \be(\S) + \sum_{\S \in T} \omega_S \be(\S)$.
\end{defi}

Remarks: 
\begin{itemize}
\item Note that when $T_0 = \emptyset$, the concept of $T_0$-balanced is equivalent to the usual balancedness concept. 
\item All results in this paper concern with finding the nucleolus. These results and the algorithms can be simplified to finding the pre-nucleolus by setting $T_0 = \emptyset$.
\end{itemize}

For any collection $Q$ of coalitions, let us define
$$Y(Q) = \left\{ \by\in \R^n~:~\by(\S) \geq 0~\forall \S \in Q,~ \by(N) = 0\right\}.$$
We have $Y(Q) \neq \emptyset$ since $\mb{0} \in Y(Q)$. The first key result in \citet{kohlberg1971nucleolus} that will be exploited in this work is the following lemma:

\begin{lemma}[\citet{kohlberg1971nucleolus}] \label{lemma:Kohlberg}
Given a collection $T_0 \subseteq 2^N$, a collection $T\subseteq 2^N$ is $T_0$-balanced if and only if $\by \in Y(T_0 \cup T)$ implies $\by(\S)=0,\forall \S \in T$.
\end{lemma}

This result allows the author to define two sets of equivalent properties on a sequence of collections $(Q_0,Q_1,\ldots)$ as:
\begin{defi}
$(Q_0,Q_1,\ldots)$ has Property I if for all $k \geq 1$, the following claim holds: $\displaystyle \by \in Y(\cup_{j=0}^k Q_j)$ implies $\displaystyle \by(\S) = 0, ~\forall \S \in \cup_{j=1}^k Q_j$.
\end{defi}

\begin{defi}
$(Q_0,Q_1,\ldots)$ has Property II if for all $k \geq 1$, $\displaystyle \cup_{j=1}^k Q_j$ is $Q_0$-balanced.
\end{defi}

The main result in \citet{kohlberg1971nucleolus} can be summarized in the following theorem:

\begin{theorem}[\citet{kohlberg1971nucleolus}] \label{theorem:Kohlberg}
The following three claims are equivalent:
(a) $\bx$ is the nucleolus; (b) $(T_0(\bx),T_1(\bx),\ldots)$ has Property I; and (c) $(T_0(\bx),T_1(\bx),\ldots)$ has Property II.
\end{theorem}

For completeness, we show a slightly different version and proof of Theorem~\ref{theorem:Kohlberg} by Lemma~\ref{lemma:balancedness2} in Appendix~A.

To appreciate the practicality of the Kohlberg criterion and for convenience in the later development, we first present the algorithmic view of the Kohlberg criterion in Algorithm~\ref{alg:Kohlberg}.

\begin{algorithm}[!h]
\textbf{Input}: Game $G(N,v)$, imputation solution $\bx$\;
\textbf{Output}: Conclude if $\bx$ is the nucleolus\;
\textbf{1}. Initialization: Set $H_0 = \{\be_N,\emptyset\}$, $T_0 = \{\{i\},i=1,\ldots,n~:~x_i = v(\{i\})\}$ and $k=1$\;
\While{$\H_{k-1} \neq 2^\N$}{
        \textbf{2}. Set $T_k = \displaystyle \argmax_{\S \not \in H_{k-1}} \left\{v(\S) - \bx(\S)\right\}$\;
        \eIf{$(\displaystyle \cup_{j=1}^k T_j)$ is $T_0$-balanced} {\textbf{3}. Set $H_k=H_{k-1} \cup T_k$, $k=k+1$ and continue} { \textbf{4}. Stop the algorithm and conclude that $\bx$  \textbf{is not} the nucleolus}
        }
\textbf{5}. Conclude that $\bx$ \textbf{is} the nucleolus.
\caption{Kohlberg Algorithm for verifying if a solution is the nucleolus of a cooperative game.}
\label{alg:Kohlberg}
\end{algorithm}

In this algorithm, we iteratively form the tight sets $T_j,j=0,1,\ldots$ until either all the coalitions are included and we conclude the input solution is the nucleolus (i.e., stopping at step 5) or stop at a point where the union of the tight coalitions is not $T_0$-balanced (in step 4), in which case we conclude that the solution is not the nucleolus. To demonstrate the Kohlberg criterion, we consider the following simple three-player cooperative game:
\begin{ex}\label{ex1}
Let the characteristic function be: $v(\{1\}) = 1,~ v(\{2\}) = 1,~ v(\{3\}) = 1$, $v(\{1,2\}) = 7$, $v(\{1,3\}) = 4$, $v(\{2,3\}) = 5$, $v(\{1,2,3\}) = 12$. The set of all imputations is: $\mathbf{I} =\{ (x_1, x_2, x_3) ~:~ x_1 + x_2+x_3 = 12,~ x_1 \geq 1,~x_2 \geq 1,~x_3 \geq 1\},$ and the core of the game is:
 \begin{eqnarray*}
\mathbf{Co} =\{ (x_1, x_2, x_3) : && x_1 + x_2+x_3 = 12,~x_1 \geq 1,~x_2 \geq 1, ~x_3 \geq 1,\\
&&x_1 + x_2 \geq 7,~x_1 + x_3 \geq 4, ~x_2 + x_3 \geq 5\}.
\end{eqnarray*}


The least core is the line segment connecting $\bx=(5,4,3)$ and $\by=(3,6,3)$. The nucleolus is $\bnu = (4,5,3)$. At the nucleolus $\bnu$, we can find $T_0 = \emptyset$, $T_1 = \{\{1,2\}\},\{3\}\}$, $T_2 = \{\{1\}\},\{2,3\},\{1,3\}\}$,  $\epsilon_1 = -2$, $\epsilon_2 = -3$, and $\epsilon_3 = -4$. Here, $T_1$ is $T_0$-balanced with a weight $\mb{\omega} = (1,1)$. Similarly, $(T_1 \cup T_2)$ is $T_0$-balanced with a weight $\mb{\omega} = (1/2,1/4,1/4,1/2,1/4)$. We can also verify that the Kohlberg criterion does not hold for any $\bx' \neq \bnu$. For example, let $\bx' = 1/2(\bx+\bnu)$. Then $T_1 = \{\{1,2\}\},\{3\}\}$ and $T_2 = \{\{2,3\}\}$. Although $(T_1)$ is $T_0$-balanced, $(T_1 \cup T_2)$ is not.

\end{ex}

\section{The simplifying Kohlberg criterion}\label{sec:simplifiedKohlberg}

The Kohlberg criterion offers a powerful tool to assess whether a given payoff distribution is the nucleolus by providing both the necessary and sufficient conditions. This often arises in relatively small or well-structured games where a potential candidate for the nucleolus can be easily identified and where checking the balancedness of the corresponding tight sets can be done easily (possibly analytically). For larger games, it is inconvenient to apply the Kohlberg criterion because this could involve forming and checking for the balancedness of an exponentially large number of subsets of tight coalitions, each of which could be of exponentially large size. This section aims to resolve these issues.

\subsection{Bounding the number of iterations to $(n-1)$}\label{subsec:simplifiedKohlberg}

On using linear algebra operators on the collection of coalitions, we slightly abuse the notations and refer each coalition $\S \in 2^{\N}$ interchangeably with its binary vector $\be(\S)$ indicating whether the players are in the coalition. For each collection of coalitions $T$, let us denote $rank(T)$ as the rank of the coalitions in $T$ and $span(T)$ as the collection of all coalitions that lie in the linear span of the coalitions in $T$. The key idea in simplifying the Kohlberg criterion is to note that, once we have obtained and verified the $T_0$-balancedness of $\cup_{j=1}^k T_j$, we do not have to be concerned about all those coalitions that belong to $span(\cup_{j=1}^k T_j)$. In brief, this is because once a collection is $T_0$-balanced, its span is also $T_0$-balanced as is formalized in the following lemma:


\begin{lemma}\label{lemma:balancedness} From any collection $T_0 \subseteq 2^\N$, the following results hold:
\begin{itemize}
\item[(a)] If a collection $T$ is $T_0$-balanced, then $span(T)$ is also $T_0$-balanced.
\item[(b)] If collections $U,V$ are $T_0$-balanced then $U \cup V$, $span(U) \cup span(V)$ are also $T_0$-balanced.
\item[(c)] If $U$ is $T_0$-balanced and $U \subseteq V$, then $span(U) \cap V$ is also $T_0$-balanced.
\end{itemize}
\end{lemma}
\begin{proof}
(a) Given that $T$ is $T_0$-balanced, there exists $\mb{\gamma} \in \R_{\geq0}^{|T_0|}$ and $\mb{\omega} \in \R_{>0}^{|T|}$ such that 
$$\be(\N) = \sum_{\S \in T_0} \gamma_S \be(\S) + \sum_{\S \in T} \omega_S \be(\S).$$

For any $\S_0 \in span(T)$, there exists $\mb{\beta}$ such that $\be({\S_0}) =\sum_{\S \in T} \beta_S \be(\S)$. Thus, for any $\delta$, we have
\begin{eqnarray*}
  \be(\N) &=& \sum_{\S \in T_0} \gamma_S \be(\S) + \sum_{\S \in T} \omega_S \be(\S)\\
  &=& \sum_{\S \in T_0} \gamma_S \be(\S) + \sum_{\S \in T} \omega_S \be(\S) + \delta (\be(\S_0) -\sum_{\S \in T} \beta_S \be(\S))\\
  &=& \delta \be(\S_0)  + \sum_{\S \in T_0} \gamma_S \be(\S) + \sum_{\S \in T} (\omega_S-\delta \beta_S) \be(\S).
\end{eqnarray*}
Since $\mb{\alpha} >0$, we can choose $\delta >0$ which is small enough such that $(\alpha_\S-\delta \beta_\S) > 0,~\forall \S \in T$. Thus, $T\cup\{\S_0\}$ is a $T_0$-balanced collection. Since this holds for all $\S_0 \in span(T)$, we can conclude that $span(T)$ is $T_0$-balanced.

(b) Given that collections $U,V$ are $T_0$-balanced, there exists $\mb{\gamma},\mb{\omega} \in \R_{\geq0}^{|T_0|}$ and $\mb{\alpha} \in \R_{>0}^{|U|}$, $\mb{\beta} \in \R_{>0}^{|V|}$ such that 
$$\be({\N}) =\sum_{\S \in T_0} \gamma_S \be(\S) +\sum_{\S \in U} \alpha_S \be(\S) = \sum_{\S \in T_0} \omega_S \be(\S) + \sum_{\S \in V} \beta_S \be(\S).$$ 
This leads to 
$$\be({\N}) = \sum_{\S \in T_0} (1/2\gamma_S + 1/2\omega_S) \be(\S) + \sum_{\S \in U} 1/2\alpha_S \be(\S) + \sum_{\S \in V} 1/2 \beta_S \be(\S).$$ 
Thus $U \cup V$ is also $T_0$-balanced. We can also prove that $span(U) \cup span(V)$ is $T_0$-balanced in a similar way as shown in the proof of part (a).

(c) The proof is similar to part (a) due to the fact that, for any $\S_0 \in span(U) \cap V$, we have $\S_0 \in span(U)$ and hence $U \cup \S_0$ is also $T_0$-balanced. Thus, $span(U) \cap V$ is $T_0$-balanced. $\qed$
\end{proof}

With this result, we can provide an improved Kohlberg criterion as shown in Algorithm~\ref{alg:Modified_Kohlberg1}.

\begin{algorithm}[!h]
\textbf{Input}: Game G(N,v), imputation solution $\bx$\;
\textbf{Output}: Conclude if $\bx$ is the nucleolus or not\;
\textbf{1}. Initialization: Set $H_0 = \{\be_N,\emptyset\}$, $T_0 = \{\{i\},i=1,\ldots,n~:~x_i = v(\{i\})\}$ and $k=1$\;
\While{$rank(H_{k-1}) < n$}{
        \textbf{2}. Find $T_k = \displaystyle \argmax_{\S \not \in span(H_{k-1})} \left\{v(\S) - \bx(\S)\right\}$\;
        \eIf{$(\displaystyle \cup_{j=1}^k T_j)$ is $T_0$-balanced}{
           \textbf{3}. Set $H_k = H_{k-1} \cup T_k$, $k=k+1$ and continue\;
         }
        { \textbf{4}. Stop the algorithm and conclude that $\bx$ \textbf{is not} the nucleolus.}
      }
\textbf{5}. Conclude that $\bx$ \textbf{is} the nucleolus.
\caption{Simplified Kohlberg Algorithm for verifying if a solution is the nucleolus of a cooperative game.}
\label{alg:Modified_Kohlberg1}
\end{algorithm}

The main differences between Algorithm~\ref{alg:Modified_Kohlberg1} and Algorithm~\ref{alg:Kohlberg} are: (a) the stopping condition of the while loop has been changed from $H_{k-1} \neq 2^\N$ to $rank(H_{k-1}) < n$, and (b) the search space at step 2 has been changed from $\S \not \in H_{k-1}$ to $\S \not \in span(H_{k-1})$. 
As a result, we have the following desirable property:

\begin{theorem}\label{thm:algo1_correctness}
The while-loop in Algorithm~\ref{alg:Modified_Kohlberg1} terminates after at most $(n-1)$ iterations and it correctly decides whether a solution is the nucleolus.
\end{theorem}


\begin{proof}
First of all, by construction in step 2 of the algorithm, $T_k \cup span(H_{k-1}) = \emptyset$ and hence, by step 3, we have the rank of $H_k = H_{k-1} \cup T_k$ keeps increasing. Therefore,
$$n \geq rank(H_k) = rank(H_{k-1} \cup T_k) \geq rank(H_{k-1})+1\geq \ldots \geq rank(H_0)+k =k+1,$$
and hence the algorithm, i.e., the while loop, terminates in at most $(n-1)$ iterations. Here, we note that the algorithm terminates at either step 4 or step 5 with complementary conclusions.

Proving that the algorithm correctly decides whether a solution is the nucleolus is equivalent to showing that (a) if $\bx$ is the nucleolus then the algorithm correctly terminates at step 5, and (b) if the algorithm terminates at step 5, then the input solution must be the nucleolus.  

Part (a): If $\bx$ is the nucleolus, then $T_1$ must be $T_0$-balanced as a direct result from the Kohlberg criterion (described in Theorem~\ref{theorem:Kohlberg}). Thus $T_1$ is $T_0$-balanced and the algorithm goes through to step 3 at $k=1$. Suppose, as a contradiction, that the algorithm goes through to step 4, instead of step 5, at some index $k>1$; that is $(\cup_{j=1}^k T_j)$ is not $T_0$-balanced. By Lemma~\ref{lemma:Kohlberg}, there exists $\by \in R^n$ such that
\small
\begin{eqnarray}
  &&\by(\S) \geq 0, \forall \S \in \cup_{j=0}^k T_j;~ \by(N) = 0;~ \by(\S') > 0, \text{ for some } \S' \in \cup_{j=1}^k T_j. \label{ieq:thm1:a0}
\end{eqnarray}
\normalsize
Notice, however, that $\cup_{j=1}^{k-1} T_j$ is $T_0$-balanced by the construction in step 3 of the previous iteration. Therefore $\S' \not \in H_{k-1}$ since, otherwise, the result in Lemma~\ref{lemma:Kohlberg} is violated. Thus $\S' \in T_k$ and hence~(\ref{ieq:thm1:a0}) leads to

\small
\begin{eqnarray*}
  &&(\bx+\by)(\S) \geq \bx(\S), \forall \S \in T_k;~ (\bx+\by)(\S') > \bx(\S'), \text{ for some } \S' \in T_k,\\
  &\Rightarrow& d(\S,\bx+\by) \leq d(\S,\bx), \forall \S \in T_k;~ d(\S',\bx+\by) < d(\S',\bx), \text{ for some } \S' \in T_k,
\end{eqnarray*}
\normalsize
that is, for all coalitions in $T_k$, the corresponding excess values for $(\bx+\by)$ is no greater than that of $\bx$ with at least one strict inequality for some coalition. Thus,
\small
\begin{eqnarray}
  && \Phi(\bx+\by) <_{L,T_{k}} \Phi(\bx),\label{ieq:thm1:b}  
\end{eqnarray}
\normalsize

where, for each collection of coalition $Q$, the subscript $(\cdot_{L,Q})$ is the lexicographical comparison with respect to \emph{only} coalitions in $Q$. Since $H_{k-1}$ is $T_0$-balanced by the construction in step 3 of the previous iteration, $span(H_{k-1})$ is also $T_0$-balanced by Lemma~\ref{lemma:balancedness}. Thus, $\by(\S) = 0,~ \forall \S \in span(H_{k-1})$ and
\begin{eqnarray}
\Phi(\bx+\by) =_{L,span(H_{k-1})} \Phi(\bx).\label{ieq:thm1:a}
\end{eqnarray}

From~\refs{ieq:thm1:b} and~\refs{ieq:thm1:a} we have
\begin{eqnarray}
\Phi(\bx+\by) <_{L,span(H_{k-1})\cup T_k} \Phi(\bx).\label{ieq:thm1:c}
\end{eqnarray}

For all $\S \not \in (span(H_{k-1})\cup T_k)$ we have $v(\S) -\bx(\S) < \epsilon_k$. Thus, there exists $\delta >0$ and small enough such that
$$v(\S) -(\bx+\delta \by)(\S) < \epsilon_k,~\forall \S \not \in (span(H_{k-1})\cup T_k).$$

Note that results in~\refs{ieq:thm1:a0} also holds if we scale $\by$ by any positive factor. Thus,
\begin{eqnarray*}
\Phi(\bx+\delta \by) <_{L,span(H_{k-1})\cup T_k} \Phi(\bx).
\end{eqnarray*}
In other words, the $|span(H_{k-1})\cup T_k|$ largest excess value of $\bx$ is lexicographically larger than the excess values of $(\bx+\delta \by)$ on these collections of coalitions and the remaining coalitions which means $\bx$ is not the nucleolus. Contradiction!

Part (b): If the algorithm bypassed step 4 and went to step 5, then, $(\cup_{j=1}^k T_j)$ is $T_0$-balanced for all $k$ until $rank(H_{k-1}) = n$. Let $\bz$ be the nucleolus; then by its definition, its worst excess value should be no larger than the worst excess value of $\bx$, which is equal to $\epsilon_1$. Thus, the excess value of $\bz$ over any coalition, including those in $T_1$, must be at most $\epsilon_1$; i.e.,
\begin{eqnarray*}
  &&(\bz-\bx)(\S) \geq 0, \forall \S \in T_{1}.
\end{eqnarray*}
Since $T_1$ is $T_0$-balanced, we have, by Lemma~\ref{lemma:Kohlberg}, $(\bz-\bx)(\S) = 0$ for all $\S \in T_1$ (by noticing also that $(\bz-\bx)(\N) = 0$ and $(\bz-\bx)(\S) \geq 0,\forall \S \in T_0$ from the fact that $\bz$ is an imputation and the construction of $T_0$). Using a similar argument, given that $\bx$ and $\bz$ are lexicographically equivalent on $span(T_1)$ and since $\bz$ is the nucleolus, we also have $(\bz-\bx)(\S) \geq 0, \forall \S \in T_2$. Thus,
\begin{eqnarray*}
  &&(\bz-\bx)(\S) \geq 0, \forall \S \in T_1 \cup T_2.
\end{eqnarray*}
Again, given that $(T_1 \cup T_{2})$ is $T_0$-balanced, we have, by Lemma~\ref{lemma:Kohlberg}, $(\bz-\bx)(\S) = 0$ for all $\S \in T_1 \cup T_2$. We can continue with this and use an induction argument to show that $(\bz-\bx)(\S) = 0$ for all $\S \in H_{k-1}, k\geq 1$. Given that $rank(H_{k-1}) = n$, we must have $\bx=\bz$ or $\bx$ is the nucleolus. $\qed$
\end{proof}
\textbf{Remark}: It is noted that step 2 in both Algorithms~\ref{alg:Kohlberg} and~\ref{alg:Modified_Kohlberg1} involves comparing vectors of exponentially large sizes. Indeed we cannot escape from having an exponentially large number of operations because we are (lexicographically) comparing exponentially large vectors. The key finding in Theorem~\ref{thm:algo1_correctness}, however, is to show that step 2 of Algorithm~\ref{alg:Modified_Kohlberg1} is not repeated more than $(n-1)$ times (instead of possibly exponential times in the original Kohlberg criterion described in Algorithm~\ref{alg:Kohlberg}). Although we cannot escape from having exponential number of operations for games without any structure, there are structured games such as the voting game, the network flow game and the coalitional skill games in which step 2 can be done efficiently. We refer the readers to~\citet{Nguyen2016} for more details on this.

\subsection{Fast algorithm for checking the balancedness}\label{subsec:balancedness_checking}
According to the Kohlberg criterion, to check the $T_0$-balancedness of $T$, we need to show the existence (or non-existence) of $\mb{\gamma} \in \R_{\geq0}^{|T_0|}$ and $\mb{\omega} \in \R_{>0}^{|T|}$ such that 
$$\be(\N) = \sum_{\S \in T_0} \gamma_S \be(\S) + \sum_{\S \in T} \omega_S \be(\S).$$
This is not a big issue for small-sized $T$ where the inspection of such $(\gamma,\omega)$ can be done easily. \citet{solymosi2015} [Lemma 3] provide an approach by solving $|T|$ linear programs as follows. For each $\C \in T$, let
$$q_{\C}^* =\left\{\max w_{\C} ~:~ \sum_{\S \in T_0 } \gamma_\S \be(\S)+\sum_{\S \in T} w_\S \be(\S) = \be(\N),~ (\mb{\gamma},\mb{\omega}) \in \R_{\geq 0}^{|T_0|+|T|} \right\}.$$
Then $T$ is $T_0$-balanced if and only if $q_{\C}^*>0,\forall {\C} \in T$. Notice, however, that the collection $T$ appearing in the Kohlberg criterion could be exponentially large, and hence solving all the $|T|$ linear programs is not practical. We present a faster approach that involves at most $rank(T)$ linear programs (this is an upper-bound and, in practice, we often need to solve a much smaller number of LPs). Algorithm~\ref{alg:balancedness} describes this in details.

\begin{algorithm}[!h]
\textbf{Input}: A collection of coalitions $T$\;
\textbf{Output}: To conclude if $T$ is $T_0$-balanced or not\;

\textbf{1}. Initialization: Set $U = \emptyset$\;
\While{$rank(U) < rank(T)$}{
        \textbf{2}. Solve
        $\quad \displaystyle (\mb{\gamma}^*,\mb{\omega}^*) = \argmax_{(\gamma,\omega) \in \R_{\geq 0}^{|T_0|+|T|}}~ \sum_{\S \in T \backslash U} \omega_\S~ \text{s.t.} ~ \sum_{\S \in T_0 } \gamma_\S \be(\S)+\sum_{\S \in T } \omega_\S \be(\S) = \be(\N)
        $\;
        \eIf{$\norm{\mb{\omega}^*} = 0$}{
           \textbf{3}. Stop the algorithm and conclude that $T$ is not $T_0$-balanced\;
         }
        { \textbf{4}. Set $U = span(U \cup \{\S:\omega_\S^*>0\}) \cap T$\; }
       }
      \textbf{5}. Stop the algorithm and conclude that $T$ is $T_0$-balanced.
\caption{Algorithm for checking the $T_0$-balancedness}
\label{alg:balancedness}
\end{algorithm}

\begin{theorem}
Algorithm~\ref{alg:balancedness} correctly decides if $T$ is $T_0$-balanced and it terminates in at most $rank(T)$ iterations.
\end{theorem}
\begin{proof}
First of all, the while loop should terminate given that $rank(U)$ keeps increasing via the construction of $U$ in steps 2 and 4; i.e., the set $U$ is kept added with coalitions outside its span. Thus, the algorithm terminates at either step 3 or 5 and we need to prove that the corresponding conclusions are correct. If the algorithm terminates at step 3, then $\norm{\mb{\omega}^*} = 0$ (as otherwise the optimal solution in step 2 should be strictly positive) and hence $T$ is not $T_0$-balanced. If the algorithm terminates at step 5 then, prior to that, we have $rank(U) = rank(T)$ in order for the while loop to terminate. The construction of $U$ in step 4 ensures that $U$ is a $T_0$-balanced set by Lemmas~\ref{lemma:balancedness}b and \ref{lemma:balancedness}c. Thus $T= span(U) \cap T$ is also $T_0$-balanced by Lemma \ref{lemma:balancedness}c. In conclusion, the algorithm always terminates with the correct conclusion. $\qed$
\end{proof}

\subsection{Reducing the sizes of the tight sets}\label{subsec:simplifiedKohlberg2}
On checking the Kohlberg criterion, we might end up having to store an exponentially large number of coalitions. We provide a method for reducing this to the size of at most $n(n-1)$. We start with the following theoretical results.


\begin{theorem}
\label{lemma:balancedness2}
The following results hold
\begin{itemize}
\item[(a)] If $T \subseteq 2^N $ is a $T_0$-balanced set then there exists $R \subseteq T$ with $1 \leq |R| = rank(R) \leq rank(T)$ that is $T_0$-balanced.
\item[(b)] For nonempty $P,Q \subseteq 2^N$ with $Q \cup P$ is a $T_0$-balanced set, there exists a subset $R \subseteq Q$ with $1 \leq |R| = rank(R) \leq rank(Q)$ such that $R \cup P$ is $T_0$-balanced.
\end{itemize}
\end{theorem}
\begin{proof}$~~~$

\begin{itemize}
\item[(a)] Given that $T$ is $T_0$-balanced, there exists $\mb{\gamma} \in \R_{\geq0}^{|T_0|}$ and $\mb{\omega} \in \R_{>0}^{|T|}$ such that 
$$\be(\N) = \sum_{\S \in T_0} \gamma_S \be(\S) + \sum_{\S \in T} \omega_S \be(\S).$$
Thus,
\begin{eqnarray*}
  &&0 \neq \frac{1}{\sum_{\S \in T} \omega_S} \left(\be(\N) - \sum_{\S \in T_0} \gamma_S \be(\S)\right) = \sum_{\S \in T} \frac{\omega_S}{\sum_{\S \in T} \omega_S} \be(\S),
\end{eqnarray*}
i.e., $\frac{1}{\sum_{\S \in T} \omega_S} \left(\be(\N) - \sum_{\S \in T_0} \gamma_S \be(\S)\right) $ belongs to the convex combination of $\{\be(\S)\}_{\S \in T}$.  Applying the Caratheodory theorem, there exists a subset $U \subseteq T$ with $rank(U) = |U|= dim(T)$ such that $\frac{1}{\sum_{\S \in T} \omega_S} \left(\be(\N) - \sum_{\S \in T_0} \gamma_S \be(\S)\right) = \sum_{\S \in U} \beta_\S \be(\S)$.

By removing those coefficients $\beta_\S =0$, we obtain a subset $R \subseteq U \subseteq T$ with $rank(R) \leq rank(U)$ that is $T_0$-balanced. Note also that, since 
$$\frac{1}{\sum_{\S \in T} \omega_S} \left(\be(\N) - \sum_{\S \in T_0} \gamma_S \be(\S)\right) \neq 0,$$
there exists at least a coalition $\S$ with $\beta_\S >0$. Thus $1 \leq rank(R) \leq rank(T)$ and $R$ is $T_0$-balanced. In addition, since $rank(U) = |U|$, we also have $rank(R) = |R|$. 

\item[(b)] Since $Q \cup P$ is $T_0$-balanced, there exists $\mb{\gamma} \in \R_{\geq0}^{|T_0|}$, $\mb{\alpha} \in \R_{>0}^{|P|}$ and $\mb{\beta} \in \R_{>0}^{|Q|}$ such that $\be({\N})= \sum_{\S \in T_0} \gamma_\S \be(\S) + \sum_{\S \in P} \alpha_\S \be(\S) + \sum_{\S \in Q} \beta_\S \be(\S)$. Thus,
\begin{eqnarray*}
  &&(\be(\N)-\sum_{\S \in T_0} \gamma_\S \be(\S)-\sum_{\S \in P} \alpha_\S \be(\S))= \sum_{\S \in Q} \beta_\S \be(\S) \neq 0.
\end{eqnarray*}
Using the same argument as in part (a), there exists a subset $Q' \subseteq Q$ with $rank(Q') = |Q'|= dim(Q)$ such that
\begin{eqnarray*}
  &&(\be(\N)-\sum_{\S \in T_0} \gamma_\S \be(\S)-\sum_{\S \in P} \alpha_\S \be(\S))= \sum_{\S \in Q'} \beta_\S \be(\S).
\end{eqnarray*}
By removing those coalitions $\S \in Q'$ with $\beta_\S=0$, we obtain a non-empty subset $R \subseteq Q'$ such that $R \cup P$ is $T_0$-balanced and $1 \leq |R| = rank(R) \leq rank(Q)$.
$\qed$
\end{itemize}
\end{proof}


%

We denote such a subset $R$ in Theorem~\ref{lemma:balancedness2}a as $R = rep(T;T_0)$ and subset $R$ in Theorem~\ref{lemma:balancedness2}b as $R = rep(Q;P,T_0)$. Algorithm~\ref{alg:Modified_Kohlberg2} shows the improved Kohlberg Algorithm for verifying if a solution is the nucleolus by replacing each tight set of coalitions by its representation derived in Theorem~\ref{lemma:balancedness2}.

\begin{algorithm}[!h]
\textbf{Input}: Game G(N,v), imputation solution $\bx$\;
\textbf{Output}: Conclude if $\bx$ is the nucleolus or not\;
\textbf{1}. Initialization: Set $H_0 = T_0 = \be_N$, $T_0 = \{\{i\},i=1,\ldots,n~:~x_i = v(\{i\})\}$, and $k=1$\;
\While{$rank(H_{k-1}) < n$}{
        \textbf{2}. Find $T_k = \displaystyle \argmax_{\S \not \in span(H_k)} v(\S) - \bx(\S)$\;
        \eIf{$(\displaystyle H_{k-1} \cup T_k)$ is $T_0$-balanced}{
           \textbf{3}. Set $R_k = rep(T_k ; H_{k-1},T_0)$, $H_k = H_{k-1} \cup R_k$, $k=k+1$ and continue\;
         }
        { \textbf{4}. Stop the algorithm and conclude that $\bx$ is not the nucleolus.}
      }
\textbf{5}. Conclude that $\bx$ is the nucleolus.
\caption{Improved Kohlberg Algorithm for verifying if a solution is the nucleolus of a cooperative game.}
\label{alg:Modified_Kohlberg2}
\end{algorithm}

The main difference between Algorithm~\ref{alg:Modified_Kohlberg2} and Algorithm~\ref{alg:Modified_Kohlberg1} is in step 3 where we set $H_k = H_{k-1} \cup rep(T_k; H_{k-1},T_0)$ instead of $H_k = H_{k-1} \cup T_k$. This means we store only a representative of $T_k$ in the subsequent rounds. The correctness of the algorithm can still be proven as presented in the following theorem.


\begin{theorem}
The while-loop in Algorithm~\ref{alg:Modified_Kohlberg2} terminates after at most $(n-1)$ iterations and it correctly decides whether a solution is the nucleolus.
\end{theorem}
\begin{proof}
After each iteration, we have $R_k \not \subseteq span(H_{k-1})$ and $rank(R_k) \geq 1$ by its construction. Therefore $rank(H_k)=rank(R_k \cup H_{k-1})$ keeps increasing and hence Algorithm~\ref{alg:Modified_Kohlberg2} terminates after at most $(n-1)$ iterations. We also note that the algorithm terminates at either step 4 or 5 with complementary conclusions.

Proving that the algorithm correctly decides whether a solution is the nucleolus is equivalent to showing that (a) if $\bx$ is the nucleolus then the Algorithm terminates at step 5, and (b) if the algorithm terminates at step 5, then the input solution must be the nucleolus. We use results from Lemma~\ref{lemma:balancedness} and Theorem~\ref{lemma:balancedness2} for this.

The proof for part (a) is still the same as that proof for Theorem 2 since the key property used in that proof was to keep $H_k$ always $T_0$-balanced. This is summarized as follows. If $\bx$ is the nucleolus then $T_1$ is $T_0$-balanced and the algorithm gets through to step 3 at $k=1$. Suppose, on contradiction, that the algorithm terminate at step 4 at some index $k>1$ with $(H_{k-1} \cup T_k)$ not $T_0$-balanced while $H_{k-1}$ is $T_0$-balanced by the construction in step 3 of the previous iteration. Then, by Lemma~\ref{lemma:Kohlberg}, there exists $\by \in R^n$ such that
\begin{eqnarray*}
  &&\by(\S) \geq 0, \forall \S \in T_0 \cup H_{k-1} \cup T_k;~ \by(N) = 0;~ \by(\S') > 0, \text{ for some } \S' \in T_k.
\end{eqnarray*}
Thus,
\begin{eqnarray*}
\Phi(\bx+\by) <_{L,span(H_{k-1})\cup T_k} \Phi(\bx).
\end{eqnarray*}

In addition, for all $\S \not \in (span(H_{k-1})\cup T_k)$ we have $v(\S) -\bx(\S) < \epsilon_k$ by the construction in step 2. Thus there exist $\delta >0$, which is small enough such that
$$v(\S) -(\bx+\delta \by)(\S) < \epsilon_k,~\forall \S \not \in (span(H_{k-1})\cup T_k)$$
and
\begin{eqnarray*}
\Phi(\bx+\delta \by) <_{L,span(H_{k-1})\cup T_k} \Phi(\bx).
\end{eqnarray*}
In other words, the $|span(H_{k-1})\cup T_k|$ largest excess value of $\bx$ is lexicographically larger than the excess values of $(\bx+\delta \by)$ on these collections of coalitions and the remaining coalitions which means $\bx$ is not the nucleolus. Contradiction.

The proof for part (b) is also the same as that in Theorem 2 where the key property of retaining the rank of $H_k$ increased throughout the algorithm is still preserved. Due to the $T_0$-balancedness of $(H_{k-1} \cup T_k)$, we can use Lemma~\ref{lemma:Kohlberg} to recursively show that $(\bz-\bx)(\S) = 0$ for all $\S \in H_{k-1} \cup T_k$ where $\bz$ is the nucleolus as follows.

Let $\by = \bz - \bx$. We have $\by(\N) =0$ and $\by(\S) \geq 0, \forall \S \in T_0$ due to the fact that $\bz$ is an imputation and the definition of $T_0$. We also have $\by(\S) \geq 0, \forall \S \in R_1$ due to the fact that $\Phi(\bz) \leq_{L,R_1} \Phi(\bx)$ (or otherwise $\bz$ is not the nucleolus).

Applying result from Lemma \ref{lemma:Kohlberg} with a note that $R_1$ is $T_0$-balanced, we have $\by(\S) = 0, \forall \S \in R_1$ which means $\bz$ and $\bx$ are lexicographically equivalent on $R_1$. We will prove by induction that $\by(\S) = 0, \forall \S \in \cup_{j=1}^k R_j$ for all indices $k$ valid in Algorithm~\ref{alg:Modified_Kohlberg2}. Suppose this indeed hold for $(k-1)$, i.e. 
$\by(\S) = 0, \forall \S \in \cup_{j=1}^{k-1} R_j$. In other words, $\bz$ and $\bx$ are lexicographically equivalent on $\cup_{j=1}^{k-1} R_j$ which is the collection of coalitions from which $\bx$ receives the worst excess values. In order for $\bz$ to be at least as good lexicographically as $\bx$, the excess values of $\bz$ on those coalition in $R_k$ must be no worst that those from $\bx$, i.e., $\by(\S) \geq 0, \forall \S \in R_k$. Applying result from Lemma \ref{lemma:Kohlberg} with a note that $\cup_{j=1}^k R_j$ is $T_0$-balanced, we must also have $\by(\S) = 0, \forall \S \in R_k$.

Since $rank(H_{k-1}) = n$, we must have $\bx=\bz$ or $\bx$ is the nucleolus. $\qed$
\end{proof}

\begin{theorem}
The collection of tight coalitions $\cup_{j=1}^k R_j$ stored by Algorithm~\ref{alg:Modified_Kohlberg2} is of size at most $n(n-1)$.
\end{theorem}
\begin{proof}
In the proof of Theorem~\ref{lemma:balancedness2}, note that each $R_j$ is constructed as a subset of another full row rank subset, its size is at most $n$ rows. Since the number of iterations involved is at most $(n-1)$, the total size of $\cup_{j=1}^k R_j$ is at most $n(n-1)$.
$\qed$
\end{proof}
\emph{Remark}: We conjecture that, under mild conditions, the size of $\cup{j=1}^k R_j$ is equal to $(n+k)$. We also conjecture that the algorithms developed can be extended to the case of finding the nucleolus within any polyhedron by replacing the set $T_0$ accordingly. We leave these explorations for future work though.

\section{Conclusion}
The Kohlberg criterion proves to be a powerful tool for verifying whether a payoff distribution is the nucleolus in relatively small games. Its application to larger games is, however, rather limited due to the need for repeatedly forming, storing and checking the balancedness of an exponentially large collection of coalitions for an exponentially large number of iterations. In this work, we simplify the Kohlberg criterion to achieve the following desirable properties: (a) the number of iteration is bounded to at most $(n-1)$, (b) the size of collections of coalitions for storage is at most $n(n-1)$. In addition, we provide a fast algorithm for checking the balancedness. It is expected that the findings will boost the applications of the Kohlberg criterion and possibly provide new directions for finding efficient algorithms to compute the nucleolus.

\section*{Acknowledge}
We thank Dr Holger Meinhardt \cite{meinhardt2017simplifying} for pointing out some of the typos in our ealier working draft of this paper (mostly on handling the $T_0$-balancedness condition which we overlooked when changing the original algorithm for finding the prenucleolus to finding the nucleolus). We however still disagree with Dr Meinhardt's very strong claim on the correctness of the well-established proof technique used in this paper (such as in Theorem~\ref{thm:algo1_correctness}).

\bibliographystyle{plainnat}
\bibliography{bibliography}

\section*{Appendix A: Alternative Proof of Kohlberg Criterion}\label{app:Kohlberg_thm_restate}
 
Let $(T_0,T)$ be two collections of coalitions. For each coalition $C \in T$, let us introduce the following primal LP:

\begin{eqnarray*}
P(C):\quad \max_{\mb{\gamma},\omega,\alpha} &&\omega_C \\
s.t. && \sum_{\S \in T_0} \gamma_\S e(\S)+\sum_{\S \in T} \omega_\S e(\S) -\alpha e(N) = 0,\\
&&\mb{\gamma} \geq 0, \omega \geq 0.
\end{eqnarray*}
 
The corresponding dual problem is:
\begin{eqnarray*}
D(C):\quad \min_{y} &&0 \\
s.t. && y^T e(\S) \geq 0,\forall \S  \in T_0 \cup T,\\
&& y^T e(\N) = 0,\\
&& y^T e(C) \geq 1.
\end{eqnarray*}

We have the following results:

\begin{lemma}
\label{lemma:balancedness2}
For any given pair of subsets $(T_0,T)$ of the powerset $2^N$, the following are equivalent
\begin{itemize}
\item[(a)] $T$ is $T_0$-balanced if any only if for all $C \in T$, the primal problem $P(C)$ is unbounded.
\item[(b)] For any $C \in T$, the primal problem $P(C)$ is unbounded if any only if the dual $D(C)$ is infeasible.
\item[(c)] The primal problem $D(C)$ is infeasible for all $C \in T$ if any only if $(T_0,T)$ has property II.
\item[(d)] $(T_0,T)$ has property II if and only if $T$ is $T_0$-balanced. 
\end{itemize}
\end{lemma}
\begin{proof}$~~~$
Result in part (d) is what we want to show and this will follows directly if we are able to show (a)-(c). We choose to show both sides of the if and only if statements in part (a)-(c) so that each of these can be viewed as stand-alone results eventhough the proof of the entire lemma only requires one direction such as (a)$\Rightarrow$(b)$\Rightarrow$(c)$\Rightarrow$(d)$\Rightarrow$(a). 
\begin{itemize}
\item[(a)] 
$\Rightarrow$  If $T$ is $T_0$-balanced, there exists weight vectors $\mb{\gamma} \in \R_{\geq 0}^{|T_0|}$,$\mb{\omega}  \in \R_{>0}^{|T|}$ such that 
$$\sum_{\S \in T_0} \gamma_\S e(\S) +\sum_{\S \in T} \omega_\S e(\S) = e(N).$$
For each $C \in T$, we have $(\mb{\gamma},\mb{\omega},1)$ is a feasible solution to $P(C)$ with an objective value of $\omega_C >0$. Since the problem is homogeneous on $(\mb{\gamma},\mb{\omega}, \alpha)$, that is for all $\Delta >0$, we have $(\Delta \mb{\gamma}, \Delta \mb{\omega}, \Delta \alpha)$ is also a feasible solution with an optimal value of $\Delta \omega_C$. Thus, the primal problem is unbounded and hence the dual problem $P(C)$ is infeasible.

$\Leftarrow$
For each $C \in T$, given the primal problem $P(C)$ is unbounded, we can pick a corresponding feasible solution $(\mb{\gamma},\mb{\omega},1)$ with a positive objective value $\omega_C$. Average out all such feasible solutions $(\mb{\gamma},\mb{\omega},1)$, one for each $C \in T$, we would obtain the average weight $(\bar{\mb{\gamma}},\bar{\mb{\omega}})$ that satisfies $\bar{\mb{\gamma}} \geq 0$, $\bar{\mb{\omega}} >0$, and
$$\sum_{\S \in T_0} \bar{\gamma}_\S e(\S) +\sum_{\S \in T} \bar{\omega}_\S e(\S) = e(N)$$ 
Thus, $T$ is $T_0$-balanced.

\item[(b)] We can see that the primal problem is alway feasible at $(\mb{\gamma}=0, \mb{\omega} = 0, \alpha = 0)$. In addition, the problem is homogeneous on $(\mb{\gamma}, \mb{\omega}, \alpha)$ and hence its optimal value is either zero or positive infinitive (unbounded). The dual problem, on the other hand, is either infeasible or with an optima value of zero. From linear programming duality results, it is easy to show that in this case, the primal is unbounded if and only if the dual is infeasible.

\item[(c)]
$\Leftarrow$ If $(T_0,T)$ has property II, we have $D(C)$ infeasible for all $C \in T$ by definition of property II.

$\Rightarrow$ If $D(C)$ infeasible for all $C \in T$ then $(T_0,T)$ must have property II since otherwise there exists a $y \in Y(T_0 \cup T)$ and a coalition $C \in T$ such that $y(C) >0$. Thus, we can scale up $y$ by an appropriate factor $\Delta$ such that $\Delta y \in Y(T_0 \cup T)$ and $\Delta y(C) \geq1$. This means the dual problem $D(C)$ is feasible. Contradiction!

$\qed$
\end{itemize}
\end{proof}

\end{document}